\documentclass[12pt,a4paper]{article}

\usepackage{t1enc}
\usepackage{amsmath}
\usepackage{amssymb}
\usepackage{amsthm}
\usepackage{url}

\def\hB{\mathcal{B}}
\def\hG{\mathcal{G}}
\def\hL{\mathcal{L}}
\def\hQ{\mathcal{Q}}

\def\ff{\mathbb{F}}
\def\gg{\mathbb{G}}
\def\pp{\mathbb{P}}
\def\qq{\mathbb{Q}}
\def\rr{\mathbb{R}}

\DeclareMathOperator{\GF}{\mathsf{GF}}
\DeclareMathOperator{\PG}{\mathsf{PG}}
\DeclareMathOperator{\ch}{\mathsf{char}}
\DeclareMathOperator{\mult}{\mathsf{mult}}

\theoremstyle{plain}
\newtheorem{thm}{Theorem}
\newtheorem{lem}[thm]{Lemma}
\newtheorem{prop}[thm]{Proposition}
\newtheorem{cor}[thm]{Corollary}
\theoremstyle{definition}
\newtheorem{defi}[thm]{Definition}
\newtheorem{ex}[thm]{Example}
\newtheorem{rem}[thm]{Remark}
\newtheorem*{not*}{Notation}

\begin{document}

\title{Lines in higgledy-piggledy position}
\author{
  {Sz}abol{cs}~L.~Fan{cs}ali\thanks{This research was partially supported by the European COST Action IC1104}\\
  \small{MTA-ELTE Geometric~and~Algebraic~Combinatorics Research~Group}
\and
  P{\'e}ter~{Sz}iklai\thanks{This research was partially supported by the Bolyai Grant.}\\
  \small{MTA-ELTE Geometric~and~Algebraic~Combinatorics Research~Group}\\
  \small{ELTE, Institute~of~Mathematics, Department~of~Computer~Science}
}

\maketitle

\begin{abstract}
In this article, we examine sets of lines in $\PG(d,\ff)$ meeting each hyperplane in a generator set of points. We prove that such a set has to contain at least $1.5d$ lines if the field $\ff$ has more than $1.5d$ elements, and at least $2d-1$ lines if the field $\ff$ is algebraically closed. We show that suitable $2d-1$ lines constitute such a set (if $|\ff|\ge2d-1$), proving that the lower bound is tight over algebraically closed fields. At last, we will see that the strong $(s,A)$ subspace designs constructed by {Guruswami and Kopparty}~\cite{VGSK} have better (smaller) parameter $A$ than one would think at first sight.
\end{abstract}

\section{Introduction}

{H{\'e}ger, Patk{\'o}s and Tak{\'a}{ts}}~\cite{HTP} hunt for a set $\hG$ of points in the projective space $\PG(d,q)$ that `determines' all hyperplanes in the sense that the intersection $\Pi\cap\hG$ is \emph{individual} for each hyperplane $\Pi$.

A little different but similar problem is to find a set $\hG$ such that each hyperplane is \emph{spanned} by the intersection $\Pi\cap\hG$. Such a `generator set' is always a `determining set' since if all the intersections $\Pi\cap\hG$ span the hyperplanes $\Pi$ then they must be individual. {H{\'e}ger, Patk{\'o}s and Tak{\'a}{ts}} thus began to examine `generator sets'. In projective planes generator sets and two-fold blocking sets are the same, since two distinct points span the line connecting these points.

\begin{defi}[Multiple blocking set]
A set $\hB$ of points in the projective space $\pp$ is a $t$-fold \emph{blocking set} with respect to hyperplanes, if each hyperplane $\Pi\subset\pp$ meets $\hB$ in at least $t$ points. One can define $t$-fold blocking sets with respect to lines, planes, etc.~similarly.
\end{defi}

The definition of the $t$-fold blocking set does not say anything more about the intersections with hyperplanes. In a projective space of dimension $d\ge3$, a $d$-fold blocking set can intersect a hyperplane $\Pi$ in such a set of $d$ points which is contained in a proper subspace of $\Pi$. Thus (in higher dimensions), a natural \emph{specialization} of multiple blocking sets would be the following. (Since in higher dimension a projective space is always over a field, we use the special notation $\PG(d,\ff)$ instead of the general $\pp$.)

\begin{defi}[Generator set]
A set $\hG$ of points in the projective space $\PG(d,\ff)$ is a \emph{generator set} with respect to hyperplanes, if each hyperplane $\Pi\subset\PG(d,\ff)$ meets $\hG$ in a `generator system' of $\Pi$, that is, $\hG\cap\Pi$ spans $\Pi$, in other words this intersection is not contained in any hyperplane of $\Pi$. (Hyperplanes of hyperplanes are subspaces in $\PG(d,\ff)$ of co-dimension two.)
\end{defi}

\begin{ex}
In a projective plane $\PG(2,q^2)$ there exist two disjoint Baer-subgeometries. These together constitute a 2-fold blocking set, and thus, a generator set consisting of $2q^2+2q+2$ points.
\end{ex}

\begin{rem}
In $\PG(d,q^d)$, $d$ disjoint subgeometries of order $q$ together constitute a $d$-fold blocking set. But it is not obvious whether this example is only a $d$-fold blocking set or it could be also a generator set (if we choose the subgeometries in a proper way).
\end{rem}

{H{\'e}ger and Tak{\'a}{ts}} had the idea to search for generator set which is the union of some disjoint lines and {Patk{\'o}s} gave an example for such a `determining set' as the union of the points of $2d+2$ distinct lines, using probabilistic method. They gave the name `higgledy-piggledy' to the property of such sets of lines. We investigate their idea.

\section{Hyperplane-generating sets of lines}

The trivial examples for multiple blocking sets are the sets of disjoint lines: If $\hB$ is the set of points of $t$ disjoint lines then $\hB$ is a $t$-fold blocking set (with respect to hyperplanes). {H{\'e}ger, Patk{\'o}s and Tak{\'a}ts}~\cite{HTP} suggested to search generator sets in such a form. (Though there can exist smaller examples.)

Sets of $k$ disjoint lines are always multiple ($k$-fold) blocking sets (with respect to hyperplanes) but not always generator sets, so the following definition is not meaningless.

\begin{defi}[Generator set of lines]
A set $\hL$ of lines is a \emph{generator set} (with respect to hyperplanes), if the set $\bigcup\hL$ of all points of the lines contained by $\hL$ is a generator set with respect to hyperplanes.
\end{defi}

From now on, we will examine sets of lines of the property above.

\subsection{Examples in projective planes}

Let $\pp$ be an arbitrary (desarguesian or not, finite or infinite) projective plane and let $\ell_1$ and $\ell_2$ be two distinct lines and let $Q=\ell_1\cap\ell_2$ denote the meeting point. Each line $\ell$ of $\pp$ not containing $Q$ meets $\ell_1$ and $\ell_2$ in two distinct points, thus, $\ell$ is generated. Lines containing $Q$ meet $\ell_1$ and $\ell_2$ only in $Q$, so they are not generated. This shows that two lines cannot be in higgledy-piggledy position.

\begin{ex}[Triangle]
Let $\ell_3$ be an arbitrary line not containing $Q$. Other lines containig $Q$ meet $\ell_3$, thus, they are also generated by $\{\ell_1,\ell_2,\ell_3\}$. Thus, three lines in general position constitute a generator set in arbitrary projective plane.
\end{ex}

\begin{rem}
If $\pp$ has only three lines through a point (i.e. $\pp$ is the Fano plane), three concurrent lines also form a generator set.
\end{rem}

In the projective plane $\PG(2,q)$, a minimal generator set of lines contains three lines and thus $3q+3$ points. Whereas two disjoint Baer subplanes (containing only $2q+2\sqrt{q}+2$ points) together also constitute a generator set (of points) with respect to lines. This example shows that there can exist generator set (of points) with respect to hyperplanes, containing less points than the smallest generator set of lines.

\subsection{Examples in projective spaces of dimension three}

Let $\ell_1$, $\ell_2$, $\ell_3$ are pairwise disjoint lines in $\PG(3,\ff)$, and let $\hQ^+_3(\ff)$ be the (unique) hyperbolic quadric containing these lines. Each plane of $\PG(3,\ff)$ which is not a tangent plane of $\hQ^+_3(\ff)$ meets these three lines in non-collinear three points, thus it is generated. Let $\ell$ denote one of the opposite lines meeting $\ell_1$, $\ell_2$ and $\ell_3$. Planes through $\ell$ containing neither $\ell_1$, nor $\ell_2$, nor $\ell_3$ meet these lines in collinear points (on the opposite line $\ell$), and thus, they are not generated.

\begin{rem} The reader can show that if these three lines are not pairwise disjoint, they cannot constitute a generator set: See the planes through the meeting point of two lines.
\end{rem}

\begin{ex}[Over $\GF(q)$ and over $\rr$ or $\qq$]
If there exists a line $\ell_4$ disjoint to the hyperbolic quadric $\hQ^+_3(\ff)$, then each plane $\Pi$ not generated by $\{\ell_1,\ell_2,\ell_3\}$ (meeting them in three collinear points) meet $\ell_4$ in a point $Q_4$ not on the line of the three collinear meeting points $Q_i=\Pi\cap\ell_i$, thus, $\Pi$ is generated by $\{\ell_1,\ell_2,\ell_3,\ell_4\}$ .
\end{ex}

The example above does not exist if the field $\ff$ is algebraically closed since in this case the hyperbolic quadric $\hQ^+_3(\ff)$ meets every lines.

\begin{ex}[Over arbitrary field]
Let $\ell_4$ and $\ell_5$ be two lines meeting the hyperbolic quadric $\hQ^+_3(\ff)$ above in such a way that there is no opposite line $\ell$ meeting both $\ell_4$ and $\ell_5$. Planes through opposite lines not meeting $\ell_4$ are generated by $\{\ell_1,\ell_2,\ell_3,\ell_4\}$ and planes through opposite lines not meeting $\ell_5$ are generated by $\{\ell_1,\ell_2,\ell_3,\ell_5\}$. Thus, $\{\ell_1,\ell_2,\ell_3,\ell_4,\ell_5\}$ is a set of lines in higgledy-piggledy position.
\end{ex}

\subsection{Lower bound over arbitrary (large enough) fields}

At first, we try to give another equivalent definition to the `higgledy-piggledy' property of generator sets of lines. The following is not an equivalent but a sufficient condition. Although, in several cases it is also a necessary condition (if we seek minimal sets of this type), thus, it could effectively be considered as an almost-equivalent.

\begin{thm}[Sufficient condition]\label{thm:suff}
If there is no subspace of co-dimension two meeting each element of the set $\hL$ of lines then $\hL$ is a generator set with respect to hyperplanes.
\end{thm}
\begin{proof}
Suppose that the set $\hL$ of lines is \emph{not} a generator set with respect to hyperplanes. Then there exists at least one hyperplane $\Pi$ that meets the elements of $\hL$ in a set $\Pi\cap\hL$ of points which is contained in a hyperplane $H$ of $\Pi$. Since $\Pi$ is a hyperplane it meets every line, thus each element of $\hL$ meets $\Pi$, but the point(s) of intersection has (have) to be contained in $H$. Thus the subspace $H$ (of co-dimension two) meets each element of $\hL$.
\end{proof}

The theorem above is a sufficient but not necessary condition. But if this condition above does not hold, then the set $\hL$ of lines could only be generator set in a very special way.

\begin{lem}\label{lem:sundiszno}
If the set $\hL$ of lines is a generator set with respect to hyperplanes and there exists a subspace $H$ of co-dimension two that meets each element of $\hL$ then $\hL$ has to contain at least as many lines as many points are contained in a projective line. (That is, $|\hL|\ge q+1$ if the field $\ff=\GF(q)$ and $\hL$ is infinite if the field $\ff$ is \emph{not finite}.)
\end{lem}
\begin{proof}
Let $\ell$ be a line not intersecting $H$. For each point $P_i\in\ell$ there exists a hyperplane $\Pi_i$ containing $H$ and meeting $P_i$. For each such hyperplane $\Pi_i$ there exists a line $\ell_i\in\hL$ that meets $\Pi_i$ not only in $H$, thus $\ell_i\subset\Pi_i$. Two distinct hyperplanes $\Pi_i$ and $\Pi_j$ intersect in $H$ thus the lines $\ell_i$ and $\ell_j$ have to be different lines.
\end{proof}

If we seek minimal size generator sets (and the field $\ff$ has more than $1.5d$ elements where $d$ is the dimension) we can suppose the condition of Theorem~\ref{thm:suff}, so we seek minimal size set of lines such that no subspace of co-dimension two meets each line.

\begin{lem}\label{lem:lower}
If the set $\hL$ of lines in $\PG(d,\ff)$ has at most $\left\lfloor\frac{d}{2}\right\rfloor+d-1$ elements then there exists a subspace $H$ of co-dimension two meeting each line in $\hL$.
\end{lem}
\begin{proof}
Let $\ell_1,\dots,\ell_{\lfloor\frac{d}{2}\rfloor}$ and $\ell_{\lfloor\frac{d}{2}\rfloor+i}$ ($1\le i\le d-1$) denote the elements of $\hL$.
There exists a subspace of dimension $2\left\lfloor\frac{d}{2}\right\rfloor-1$ containing the lines $\ell_1,\dots,\ell_{\lfloor\frac{d}{2}\rfloor}$ (if these lines are contained in a less dimensional subspace, it can be extended). If $d$ is even, this subspace is a hyperplane $\Pi$. If $d$ is odd, this subspace has co-dimension two, and thus it can be extended to a hyperplane $\Pi$. The hyperplane $\Pi$ meets each line, thus let $P_i\in\Pi\cap\ell_{\lfloor\frac{d}{2}\rfloor+i}$ for $i=1,\dots,d-1$. There exists a hyperplane $H$ of $\Pi$ that contains each point $P_i$ above. (If these points would be not in general position, that is not a problem.) The subspace $H$ has co-dimension two in $\PG(d,\ff)$ and it meets the lines $\ell_1,\dots,\ell_{\lfloor\frac{d}{2}\rfloor}$ since these lines are contained in $\Pi$ and $H$ is a hyperplane of $\Pi$, and $H$ meets the other lines since the meeting points are the points $P_i$.
\end{proof}

\begin{thm}[Lower bound]
If the field $\ff$ has at least $\left\lfloor\frac{d}{2}\right\rfloor+d$ elements, then a generator set $\hL$ of lines in $\PG(d,\ff)$ has to contain at least $\left\lfloor\frac{d}{2}\right\rfloor+d$ elements.
\end{thm}
\begin{proof}
Lemma~\ref{lem:sundiszno} and Lemma~\ref{lem:lower} together give the result.
\end{proof}

The examples in $\PG(2,q)$ and $\PG(3,q)$ show that this lower bound is tight in small dimensions ($d\le3$) over finite fields, and over $\rr$ and over $\qq$.

\begin{rem}
As in $\PG(2,2)$ three lines through a point are also in `higgledy-piggledy' position, four proper lines having a common transversal line meeting them can be in higgledy-piggledy position in $\PG(3,3)$.
\end{rem}

\section{Gra{ss}mann varieties}

The sufficient condition is an intersection-property of some subspaces. Such properties can naturally be handled using Gra{ss}mann varieties and Pl{\"u}cker co-ordinates. The original (hyperplane generating) property can also be translated to the language of Pl{\"u}cker co-ordinates.

Let $\gg(m,n,\ff)$ or simply $\gg(m,n)$ denote the Gra{ss}mannian of the linear subspaces of dimension $m$ and co-dimension $n$ in the vector space $\ff^{m+n}$, or, in other aspect $\gg(m,n)$ is the set of all projective subspaces of dimension $m-1$ (and co-dimension $n$) in $\PG(m+n-1,\ff)$. Via `Pl{\"u}cker embedding' we can identify this Gra{ss}mannian to the set of one dimensional linear subspaces of $\bigwedge^m\ff^{m+n}$ generated by totally decomposable multivectors, that is, $\gg(m,n)\subset\PG\left(\bigwedge^m\ff^{m+n}\right)\equiv\PG(\binom{m+n}{m}-1,\ff)$ is an algebraic variety of dimension $mn$.

The canonical isomorphism $\bigwedge^m\ff^{m+n}\equiv\bigwedge^n\ff^{m+n}$ defines a bijection between $\gg(m,n)$ and $\gg(n,m)$. Thus, the Gra{ss}mannian of subspaces of co-dimension two can be considered as the Gra{ss}mannian of the lines of the dual projective space.

\begin{rem}
If $m=2$ or $n=2$ then the Pl{\"u}cker co-ordinate vectors can be considered as alternating matrices:
$L_{ij}=a_ib_j-a_jb_i$ where $L=a\wedge b$.
\end{rem}

\begin{prop}
Let $\{L(1),\dots,L(k)\}$ denote the set of the Pl{\"u}cker \mbox{co-or}\-di\-nate vectors representing the elements of the set $\hL$ of $k$ lines in $\PG(d,\ff)$. There exists a subspace $H$ of co-dimension two in $\PG(d,\ff)$ meeting each element of $\hL$ \emph{if and only if} the subspace $L(1)^{\bot}\cap\dots\cap L(k)^{\bot}\le\PG(\binom{d+1}{2}-1,\ff)$ meets the Gra{ss}mann variety $\gg(d-1,2)$, that is, the equation system
\begin{align*}
\sum_{i<j}L_{ij}(1)H_{ij}&=0&\sum_{i<j}L_{ij}(2)H_{ij}&=0&\dots&&\sum_{i<j}L_{ij}(k)H_{ij}&=0
\end{align*}
together with the quadratic Pl{\"u}cker relations (for each quadruple $i_1i_2i_3i_4$ of indices)
\begin{equation*}
H_{i_1i_2}H_{i_3i_4}-H_{i_1i_3}H_{i_2i_4}+H_{i_1i_4}H_{i_2i_3}=0
\end{equation*}
has nontrivial solutions for $H_{ij}$.
\end{prop}
\begin{proof}
According to~\cite[Theorem~3.1.6.]{SCH}, the Pl{\"u}cker relations completely determine the Gra{ss}mannian (moreover, they generate the ideal of polynomials vanishing on it). In case $n=2$, the Pl{\"u}cker relations found in~\cite[Subsection~3.1.3.]{SCH} reduces to the form $H_{i_1i_2}H_{i_3i_4}-H_{i_1i_3}H_{i_2i_4}+H_{i_1i_4}H_{i_2i_3}=0$ for the quadruples $i_1i_2i_3i_4$ of indices. Since we consider the Gra{ss}mannian $\gg(d-1,2)$ of subspaces of co-dimension two as the Gra{ss}mannian $\gg(2,d-1)$ of lines of the dual space, the Pl{\"u}cker relations determining $\gg(d-1,2)$ are the same (using dual co-ordinates).

Let $a,b\in\ff^{d+1}$ be the homogeneous co-ordinate vectors of two projective points in $\PG(d,\ff)$ and let $x,y\in\ff^{d+1}$ be the homogeneous (dual) co-ordinate vectors of two hyperplanes in $\PG(d,\ff)$. The line connecting $\pp(a)$ and $\pp(b)$ is defined by the Pl{\"u}cker co-ordinate vector $a\wedge b\in\gg(2,d-1)$. The subspace of co-dimension two defined by the Pl{\"u}cker co-ordinate vector $x\wedge y\in\gg(d-1,2)$ is the intersection of the hyperplanes $x^{\bot}$ and $y^{\bot}$.

The line co-ordinatized by $L=a\wedge b$ and the subspace co-ordinatized by $H=x\wedge y$ meet each other if and only if the scalar product $\langle x\wedge y|a\wedge b\rangle=\langle x|a\rangle\langle y|b\rangle-\langle x|b\rangle\langle y|a\rangle$ equals to zero.

Finally, $\sum_{i<j}H_{ij}L_{ij}=\sum_{i<j}(a_ib_j-a_jb_i)(x_iy_j-x_jy_i)=\sum_{i\neq j}(a_ix_i)(b_jy_j)-\sum_{i\neq j}(a_jy_j)(b_ix_i)=\langle x|a\rangle\langle y|b\rangle-\langle x|b\rangle\langle y|a\rangle=\langle x\wedge y|a\wedge b\rangle$.
\end{proof}

\subsection{Tangents of the moment curve}\label{subs:tang}

Let $\{(1,t,t^2,\dots,t^d):t\in\ff\}\cup\{(0,0,0,\dots,1)\}\subset\PG(d,\ff)$ be the moment curve (rational normal curve) and let $\ell_t$ denote its tangent line in the point $(1,t,t^2,\dots,t^d)$, and $\ell_{\infty}$ is the tangent line in the point $(0,\dots,0,1)$ at infinity.

At first, compute the Pl{\"u}cker co-ordinates of these tangent lines. The Pl{\"u}cker co-ordinate vector of $\ell_t$ is $L(t)=a(t)\wedge\big(a(t)+\dot{a}(t)\big)=a(t)\wedge\dot{a}(t)$ where $a(t)=(1,t,t^2,t^3\dots,t^d)$ is the point of the curve ($a_i(t)=t^i$) and its derivate $\dot{a}(t)=(0,1,2t,3t^2\dots,dt^{d-1})$ is the direction (the ideal point in infinity) of the tangent line $\ell_t$. In matrix representation:
\begin{equation*}
L(t)=\left[\begin{matrix}
\phantom{-}0\phantom{t} & 1 & 2t &\ldots&(d\!-\!1)t^{d-2} & \phantom{(-1)}dt^{d-1}\\
-1\phantom{t} & 0 & t^2 &\ldots& (d\!-\!2)t^{d-1} & (d\!-\!1)t^{d\phantom{-1}}\\
-2t & -t^2 & 0 &\ldots& (d\!-\!3)t^{d\phantom{-1}} & (d\!-\!2)t^{d+1}\\
\vdots&\vdots&\vdots&\ddots&\vdots&\vdots\\
(1\!-\!d)t^{d-2}&(2\!-\!d)t^{d-1}&(3\!-\!d)t^{d\phantom{+1}}&\ldots& 0& t^{2d-2}\\
\phantom{1}(-d)t^{d-1}&(1\!-\!d)t^{d\phantom{-1}}& (2\!-\!d)t^{d+1}&\ldots&-t^{2d-2}&0\\
\end{matrix}\right]
\end{equation*}
That is, $L_{ij}(t)=a_i(t)\dot{a}_j(t)-\dot{a}_i(t)a_j(t)=t^{i}jt^{j-1}-t^{j}it^{i-1}=(j-i)t^{i+j-1}$ where $0\le i,j\le d$.

\begin{rem}
One can see that in suitable positions the Pl{\"u}cker co-ordinate vector $L(t)$ has the co-ordinates: $1,t^2,t^4,t^6,\dots,t^{2d-2}$ and the co-ordinates: $2t,2t^3,\dots,2t^{2d-3}$, thus, if $\ch\ff\neq2$, then the set $\{L(t_i):i=0,\dots,2d-2\}$ is linearly independent ($t_i\neq t_j$ if $i\neq j$).
\end{rem}

\begin{lem}\label{lem:ind}
If either $\ch\ff=p>d$ and $|\ff|\ge2d-1$ or $\ch\ff=0$, then there does \emph{not} exist any subspace of co-dimension two meeting each tangent line $\ell_t$ of the moment curve.
\end{lem}
\begin{proof}
Suppose to the contrary that there exists a subspace $H$ of {co-di}\-men\-sion two meeting each tangent line $\ell_t$. Let $H_{ij}$ ($0\le i<j\le d$) denote the (dual) Pl{\"u}cker co-ordinates of $H$. For these Pl{\"u}cker co-ordinates we have Pl{\"u}cker relations $H_{i_1i_2}H_{i_3i_4}-H_{i_1i_3}H_{i_2i_4}+H_{i_1i_4}H_{i_2i_3}=0$ for all quadruple $i_1i_2i_3i_4$ of indices.

The indirect assumpion means that $\sum_{i<j}H_{ij}L_{ij}(t)=0$ for all $t\in\ff$.
\begin{align*}
\sum_{i<j}H_{ij}L_{ij}(t)=\sum_{i=0}^{d-1}\sum_{j=i+1}^{d}H_{ij}(j-i)t^{i+j-1}&=
\sum_{N=1}^{d}t^{N-1}\sum_{i=0}^{\lfloor\frac{N}{2}\rfloor}(N-2i)H_{i,N-i}\quad+\\
&+\sum_{N=d+1}^{2d-1}t^{N-1}\sum_{i=1}^{d-\lfloor\frac{N}{2}\rfloor}(N-2i)H_{i,N-i}
\end{align*}
Since the field $\ff$ has more than $2d-2$ elements, this polynomial above can vanish on each element of $\ff$ only if $\sum_{i}(N-2i)H_{i,N-i}=0$ for all $N<2d$. So we have $2d-1$ new (linear) equations for the Pl{\"u}cker co-ordinates:
\begin{align*}
H_{0,1}&=0\tag{1}\\
2H_{0,2}&=0\tag{2}\\
3H_{0,3}+H_{1,2}&=0\tag{3}\\
4H_{0,4}+2H_{1,3}&=0\tag{4}\\
5H_{0,5}+3H_{1,4}+H_{2,3}&=0\tag{5}\\
6H_{0,6}+4H_{1,5}+2H_{2,4}&=0\tag{6}\\
&\;\;\vdots\\
dH_{0,d}+(d-2)H_{1,d-1}+\dots+\left(\left\lceil\tfrac{d}{2}\right\rceil-\left\lfloor\tfrac{d}{2}\right\rfloor\right)H_{\lfloor\frac{d}{2}\rfloor,\lceil\frac{d}{2}\rceil}&=0\mbox{\tag{$d$}}\\
&\;\;\vdots\\
3H_{d-3,d}+H_{d-2,d-1}&=0\mbox{\tag{$2d-3$}}\\
2H_{d-2,d}&=0\mbox{\tag{$2d-2$}}\\
H_{d-1,d}&=0\mbox{\tag{$2d-1$}}\\
\end{align*}
Notice that in equations (1), (2),\dots, ($N$), the Pl{\"u}cker co-ordinates $H_{ij}$ occur with indices $0\le i<j\le N-i$, if $N<d$. Similarly, in equations ($2d-1$), ($2d-2$), \dots, ($2d-N$) the Pl{\"u}cker co-ordinates occur with indices \mbox{$2d-N-j\le i<j\le d$}, if $N<d$.

Using these equations and the Pl{\"u}cker relatios, we can prove by induction, that all Pl{\"u}cker co-ordinates $H_{ij}$ are zero, and thus, they are not the homogeneous co-ordinates of any subspace $H$. We do two inductions, one for $N=1,\dots,d$ (increasing) and another (decreasing) one for $N'=(2d-N)=2d-1,\dots,d+1$. Remember that $\ch\ff=0$ or $\ch\ff>d$, so the nonzero integers in these equations are nonzero elements of the prime field of $\ff$.

\paragraph{Increasing induction}
The first two equations say that $H_{01}=H_{02}=0$. Suppose by induction that we have $H_{ij}=0$ for each pair $(i,j)$ where $0\le i<j \le N-i$, where $N$ is a positive integer less than $d$. Using this assumption, we prove that $H_{0,N+1}=H_{1,N}=H_{2,N-1}=\dots=0$, and thus $H_{ij}=0$ for each pair $(i,j)$ where $0\le i<j \le N+1-i$.

Equation~($N+1$) says that a linear combination of $H_{0,N+1}$, $H_{1,N}$, $H_{2,N-1}$, \dots, $H_{\lfloor\frac{N+1}{2}\rfloor,\lceil\frac{N+1}{2}\rceil}$ is zero. Let $H_{ij}$ and $H_{kl}$ be two arbitrary element among these above. We have the Pl{\"u}cker relation $H_{ij}H_{kl}-H_{ik}H_{jl}+H_{il}H_{jk}=0$. Using the assumption $H_{ij}=0$ for $i<j\le N-i$, this Pl{\"u}cker relation is reduced to $H_{ij}H_{kl}=0$.

Thus, these Pl{\"u}cker relations say that all $H_{ij}$ (among $H_{0,N+1}$, $H_{1,N}$, \dots, $H_{\lfloor\frac{N+1}{2}\rfloor,\lceil\frac{N+1}{2}\rceil}$) should be zero except one. And the linear Equation~($N+1$) says that this one cannot be exception either.

\paragraph{Decreasing induction}
The decreasing induction, started with the last two equations $H_{d-1,1}=H_{d-2,d}=0$ is similar.

\paragraph{}
So we have proved that each Pl{\"u}cker co-ordinate of the subspace $H$ of co-dimension two should be zero, that is a contradiction, since Pl{\"u}cker co-ordinates are homogeneous.
\end{proof}

\begin{thm}
If either $\ch\ff=p>d$ and $|\ff|\ge2d-1$ or $\ch\ff=0$, then arbitrary $2d-1$ distinct tangent lines $\ell_t$ together constitute a generator set with respect to hyperplanes.
\end{thm}
\begin{proof}
Let $\{\ell_{t_i}:i=1,2,\dots,2d-1\}$ be an arbitrary set of $2d-1$ tangent lines of the rational normal curve. It is enough to prove that there is no subspace $H$ of co-dimension two meeting each element of this set.

Suppose to the contrary that there exists such a subspace $H$ and let $H_{ij}$ be the Pl{\"u}cker co-ordinates of it. Since $H$ meets each line $\ell_{t_i}$, this means $\sum_{i<j}H_{ij}L_{ij}(t_k)=0$ for all $t_k$, $k=1,\dots,2d-1$. Thus, the polynomial $\sum_{i=0}^{d-1}\sum_{j=i+1}^{d}H_{ij}(j-i)t^{i+j-1}$ has $2d-1$ roots, but its degree is at most $2d-2$. So, if there exists such a subspace $H$ of co-dimension two, the polynomial above is the zero polynomial, and thus, $H$ meets each tangent line $\ell_t$, contradicting Lemma~\ref{lem:ind}.
\end{proof}

These results above require the characteristic $\ch\ff$ to be greater than the dimension $d$ (or to be zero). However, we can generalize these results over small prime characteristics.

\subsection{Small prime characteristics: `diverted tangents'}\label{subs:div}

The only weakness of the proof of Lemma~\ref{lem:ind} (which can be ruined by small prime characteristic) is the linear equation system for the Pl{\"u}cker co-ordinates $H_{ij}$. The Pl{\"u}cker co-ordinate $H_{ij}$ has coefficient $j-i\mod{p}$ and this could be zero for $j\neq i$ if the characteristic $p$ is not greater than the dimension $d$.

\begin{rem}
If the characteristic of $\ff$ equals to the dimension $d$, then there exists exactly one subspace of co-dimension two that meets each tangent $\ell_t$ of the moment curve. The Pl{\"u}cker co-ordinates of this subspace should be all zero except one: $H_{0,d}$. This subspace $H$ thus can be get as the intersection of two hyperplanes co-ordinatized by $[1,0,\dots,0]$ (the ideal hyperplane) and $[0,\dots,0,1]$.

In higher dimension there will be more such subspaces, and thus, their intersection is a subspace of codimension more than two, meeting each tangent line.
\end{rem}

If we substitute the coefficients $(j-i)$ by nonzero elements, the proof of Lemma~\ref{lem:ind} will be valid over arbitrary characteristic. Remember that the Pl{\"u}cker co-ordinates of the tangent line $\ell_t$ are $L_{ij}(t)=(j-i)t^{i+j-1}$ and the coefficient $(j-i)$ comes from here.

\begin{not*}
Let $\varphi:\{0,1,\dots,d\}\rightarrow\ff$ be an arbitrary \emph{injection}. If $|\ff|\le d$, such an injection there does not exist, but, if $\ff$ has more than $d$ elements, such a $\varphi$ \emph{does} exist, independently from the characteristic. For convenience sake, we suppose that $\varphi(0)=0$ and $\varphi(1)=1$.

Let $a(t)=(1,t,t^2,\dots,t^d)$ again denote the affine points of the moment curve ($a_i(t)=t^i$), and let $b(t)=(0,1,\varphi(2)t,\dots,\varphi(d)t^{d-1})$ denote the points of a special curve in the ideal hyperplane, defined by $b_j(t)=\varphi(j)t^{j-1}$.
\end{not*}

\begin{defi}[Diverted tangent lines]
Consider the line $\ell'_t$ connecting $a(t)$ and $b(t)$ instead of the tangent line $\ell_t$ of the moment curve in the point $a(t)$. The Pl{\"u}cker co-ordinate vector of the `\emph{diverted} tangent line' $\ell'_t$ is $L'(t)=a(t)\wedge b(t)$.
\begin{equation*}
L'_{ij}(t)=a_i(t)b_j(t)-b_i(t)a_j(t)=\big(\varphi(j)-\varphi(i)\big)t^{i+j-1}
\end{equation*}
Diverted tangent lines depend on the injection $\varphi$.
\end{defi}

\begin{rem}
If $\ch\ff$ is zero, the injection $\varphi$ can be the identity, and if $\ch\ff=p>d$, the injection $\varphi$ can be defined by $\varphi(k)\equiv k\mod{p}$. In these cases the diverted tangent line $\ell'_t$ determined by $\varphi$ equals to the actual tangent line $\ell_t$ of the moment curve.
\end{rem}

\begin{thm}
If $|\ff|\ge2d-1$, then arbitrary $2d-1$ distinct \emph{diverted tangent lines} $\ell'_{t_1},\dots,\ell'_{t_{2d-1}}$ (determined by arbitrary injection $\varphi$) together constitute a generator set with respect to hyperplanes.
\end{thm}
\begin{proof}
Suppose to the contrary that the subspace $H$ meets the diverted tangent lines $\ell'_{t_1},\dots,\ell'_{t_{2d-1}}$, that is, $\sum_{i<j}H_{ij}L'_{ij}(t_k)=0$ for all $k=1,\dots,2d-1$. Thus, the polynomial $\sum_{i=0}^{d-1}\sum_{j=i+1}^{d}H_{ij}\big(\varphi(j)-\varphi(i)\big)t^{i+j-1}$ has $2d-1$ roots, but its degree is at most $2d-2$. So, the polynomial above is the zero polynomial, and thus, $H$ meets each connecting line $\ell'_t$ ($t\in\ff$), that is,
\begin{align*}
\sum_{i<j}H_{ij}L_{ij}(t)=\sum_{i=0}^{d-1}\sum_{j=i+1}^{d}H_{ij}\big(\varphi(j)-\varphi(i)\big)t^{i+j-1}&=0&\forall
t\in\ff
\end{align*}
Now, we can repeat the proof of Lemma~\ref{lem:ind} by substituting the coefficients $(j-i)$ by $\big(\varphi(j)-\varphi(i)\big)$ in the linear equations (1), (2), \dots, ($2d-1$), and since $\varphi$ is injective, these coefficients are nonzero. Thus, we can prove that each Pl{\"u}cker co-ordinate $H_{ij}$ should be zero, which is a contradiction.
\end{proof}

We have proved that over arbitrary (large enough) field we can construct a hyperplane-generating set of lines of size $2d-1$. In the next section, we will prove that it is the smallest one if the field is algebraically closed.

\subsection{Lower bound over algebraically closed fields}

Over an algebraically closed field, the set $\hL$ of lines could be a generator set only if the condition of Theorem~\ref{thm:suff} holds.

\begin{lem}
{\rm\cite[Corollary~3.2.14 and Subsection~3.1.1]{SCH}}
The dimension of the Gra{ss}mannian as an algebraic variety is $\dim\gg(m,n)=mn$ and its degree is
\begin{equation*}
\deg\gg(m,n)=\frac{0!1!\dots(n\!-\!1)!}{m!(m\!+\!1)!\dots(m\!+\!n\!-\!1)!}\big(mn\big)!
\end{equation*}
In particular, the Gra{ss}mann variety $\gg(2,d-1)$ of the lines of $\PG(d,\ff)$ has dimension $2(d-1)=2d-2$ and its degree is $\frac{1}{2d-1}\binom{2d-1}{d}>0$.\qed
\end{lem}

Remember that an algebraic surface $\gg\subset\pp$ of \emph{dimension $n$} and a projective subspace $S\le\pp$ of \mbox{\emph{co-dimension $n$}} always meet over an algebraically closed field.

\begin{thm}
Over algebraically closed field $\ff$, if the set $\hL$ of lines in $\PG(d,\ff)$ has at most $2d-2$ elements, then there exists a subspace $H$ in $\PG(d,\ff)$ of co-dimension two that meets each element of $\hL$, and thus, $\hL$ is \emph{not} a generator set.
\end{thm}
\begin{proof}
Suppose that $\hL=\{L(1),\dots,L(2d-2)\}$ has exactly $2d-2$ elements (if not, we can extend it). The subspace $L(1)^{\bot}\cap\dots\cap L({2d-2})^{\bot}$ has co-di\-men\-sion at most $2d-2$ in  $\PG(\binom{d+1}{2}-1,\ff)$. The Gra{ss}mannian $\gg(d-1,2)$ of the 2-co-dimensional subspaces of $\PG(d,\ff)$ has dimension $2(d-1)=2d-2$
and its degree is $\frac{1}{2d-1}\binom{2d-1}{d}>0$.

Thus, $L(1)^{\bot}\cap\dots\cap L({2d-2})^{\bot}\cap\gg(d-1,2)$ contains at least $\frac{1}{2d-1}\binom{2d-1}{d}\ge1$ elements, which are subspaces of co-dimension two meeting the lines in $\hL$.
\end{proof}

\begin{cor}
Over algebraically closed field $\ff$, arbitrary $2d-1$ distinct \emph{diverted} tangent lines $\ell'_t$ in $\PG(d,\ff)$ constitute a generator set of minimal size. Thus, over algebraically closed fields the lower bound $2d-1$ is tight.
\end{cor}

\section{The Guruswami--Kopparty constructions}

In their very recent work~\cite{VGSK}, {Venkatesan Guruswami and Swastik Kopparty} construct subspace designs.

\begin{defi}[Weak subspace design]
{\rm\cite[Definition~2]{VGSK}}
A collection of subspaces $H_1,\dots,H_M\subset\ff_q^{d+1}$ is called a weak $(s,A)$ subspace design if for every $q$-linear subspace $W\subset\ff_q^{d+1}$ of dimension $s$, the number of indices $i$ for which $\dim_q(H_i\cap W)>0$ is at most $A$.
\end{defi}

A collection of at most $A$ subspaces would always be a weak $(s,A)$ subspace design, so the definition is not meaningless only if the subspace design contains at least $A+1$ subspaces.

\begin{defi}[Strong subspace design]
{\rm\cite[Definition~3]{VGSK}}
A collection of subspaces $H_1,\dots,H_M\subset\ff_q^{d+1}$ is called a strong $(s,A)$ subspace design if for every $q$-linear subspace $W\subset\ff_q^{d+1}$ of dimension $s$, the sum $\sum_{i=1}^M\dim_q(H_i\cap W)$ is at most $A$.
\end{defi}

Every strong $(s,A)$ subspace design is also a weak $(s,A)$ subspace design, and every weak $(s,A)$ subspace design is also a strong $(s,sA)$ subspace design. The main theorem of~\cite{VGSK} is the following.

\begin{thm}[Guruswami--Kopparty]\label{GKmain}
{\rm\cite[Theorem~7]{VGSK}}
For all positive integers $s,r,t,m=d+1$ and prime powers $q$ satisfying $s\le t\le d+1<q$, there is an explicit collection of $M=\Omega\left(\frac{q^r}{rt}\right)$ linear subspaces $H_1,\dots,H_M\subset\ff_q^{d+1}$, each of co-dimension $rt$, which forms a strong $\left(s,\frac{d\cdot s}{r\cdot (t-s+1)}\right)$ subspace design.
\end{thm}

\subsection{Relation to higgledy-piggledy lines}

If we dualize our problem (to find a minimal collection of lines such that no subspace of co-dimension two intersects all of them) and use linear terminology instead of projective one, we want to find a collection of subspaces $L_1,\dots,L_N$ of co-dimension two having the property that for every 2-dimensional subspace (projective line) $H$, at most $N\!-\!1$ of the $L_i$'s intersect $H$ non-trivially. So, we seek a weak $(2,N\!-\!1)$ subspace design of $N$ subspaces of co-dimension two, where $N$ is minimal.

\begin{rem}
If we have a weak $(s,A)$ subspace-design of $M$ subspaces \mbox{($M>A$)}, then any $A+1$ subspaces among them constitute a weak $(s,A)$ subspace design. Thus, if we have a weak $(2,N\!-\!1)$ subspace design of $M\ge N$ subspaces of co-dimension two, we will also have a set of $N$ lines in higgledy-piggledy position.
\end{rem}

We are interested in $(2,N\!-\!1)$ subspace designs containing subspaces of co-dimension two, thus $s=2=rt$, and thus $r=1$ and $t=2$. In this case the Guruswami--Kopparty~Theorem~\ref{GKmain} gives a strong $(2,2d)$ subspace design containing $M>\mathrm{const}\cdot q$ subspaces of co-dimension two. If $M>2d$, this design (after dualization) gives us a set of $2d+1$ lines in higgledy-piggledy position.

Watching the Guruswami--Kopparty constructions~\cite[Sections~4--5]{VGSK} with both eyes, we can behold the fact that these constructions yield a little bit stronger version of Theorem~\ref{GKmain}. This will be shown in the following two subsections.

\subsection{Construction of~\cite[Section~4]{VGSK}}

The main result of~\cite{VGSK} is based on the following construction. We will use $d$ instead of $m-1$. Let $s\le t\le d+1<q$ and $r$ be positive integer parameters and identify $\ff_q^{d+1}$ with the $\ff_q$-linear subspace of polynomials of degree $\le d$ in $\ff_q[X]$ and let $\omega$ denote a generator of $\ff_q^*$. For $\alpha\in\ff_{q^r}$, let $S_\alpha\subseteq\ff_{q^r}$ be given by
\begin{equation*}
S_\alpha=\{\alpha^{q^j}\omega^i\,|\,0\le j<r,0\le i<t\}.
\end{equation*}
Let $\mathcal{F}\subseteq\ff_{q^r}$ be a large set such that:
\begin{itemize}
\item For each $\alpha\in\mathcal{F}$: $\ff_q(\alpha)=\ff_{q^r}$.
\item For $\alpha\neq\beta\in\mathcal{F}$: $S_\alpha\cap S_\beta=\emptyset$.
\item Each $S_\alpha$ has cardinality $rt$.
\end{itemize}
For each $\alpha\in\mathcal{F}$ let
\begin{equation*}
H_\alpha=\{P(X)\in\ff_q^{d+1}\,|\,P(\alpha\cdot\omega^i)=0:\forall i=0,1,\dots,t-1\}
\end{equation*}

\begin{thm}[Guruswami--Kopparty]\label{thm:GKRS}
{\rm\cite[Theorem~14]{VGSK}}
Using the notation above, the collection $\{H_\alpha|\alpha\in\mathcal{F}\}$ is a strong $\left(s,\frac{d\cdot s}{r\cdot (t-s+1)}\right)$ subspace design.
\end{thm}

We do not repeat the proof here, for details see~\cite[pages~8--10]{VGSK}. The keystone of the proof of this theorem above is the following matrix. Let $W\le\ff^{d+1}_q$ be a subspace and let the polynomials $P_1,\dots,P_s$ constitute a basis of $W$. Define the following $t\times s$ matrix of polynomials:
\begin{equation*}
M(X)=\left[\begin{matrix}
P_1(X) & \dots & P_s(X)\\
P_1(X\omega) & \dots & P_s(X\omega)\\
\vdots & \ddots & \vdots \\
P_1(X\omega^{t-1}) & \dots & P_s(X\omega^{t-1})\\
\end{matrix}\right]
\end{equation*}
Let $A(X)$ be the top $s\times s$ submatrix of $M(X)$ and let $L(X)$ be the determinant of $A(X)$.

The term $d\cdot s$ in the parameter $\left(s,\frac{d\cdot s}{r(t-s+1)}\right)$ comes directly from the fact that the polynomial $L(X)$ has degree at most $d\cdot s$. We can give a better bound for this degree:

\begin{lem}
The polynomial $L(X)$ has degree at most $ds-\binom{s}{2}$.
\end{lem}
\begin{proof}
The basis $P_1,\dots,P_s$ of the subspace $W\le\ff^{d+1}_q$ can be chosen (by Gau{ss}ian elimination) such that $\deg(P_1)<\deg(P_2)<\dots<\deg(P_s)\le d$ and thus, $\deg(L)\le d+\dots+\big(d-(s-1)\big)=ds-\frac{s(s-1)}{2}$.
\end{proof}

As a consequence, the Guruswami--Kopparty Theorem~\ref{thm:GKRS} above will have the following improved form.

\begin{cor}[Guruswami--Kopparty; improved version]
Using the notation above, the collection $\{H_\alpha|\alpha\in\mathcal{F}\}$ is a strong $\left(s,\frac{\left(d-\frac{s-1}{2}\right)s}{r(t-s+1)}\right)$ subspace design.
\end{cor}

This observation shows that the Guruswami--Kopparty construction of~\cite[Section~4]{VGSK} based on Folded Reed--Solomon codes actually give us a strong $(2,2d\!-\!1)$ subspace design, and thus, a set of $2d$ lines in higgledy-piggledy position.

\subsection{Construction of~\cite[Section~5]{VGSK}}

The main result of~\cite{VGSK} is also proved by the following construction which could be used only over large characteristics. We will again use $d$ instead of $m-1$. Let $0<s\le t\le d+1<\ch\ff_q$ be integer parameters and identify $\ff_q^{d+1}$ with the $\ff_q$-linear subspace of polynomials of degree $\le d$ in $\ff_q[X]$. For each $\alpha\in\ff_q$ let
\begin{equation*}
H_\alpha=\{P(X)\in\ff_q^{d+1}|\mult(P,\alpha)\ge t\}
\end{equation*}

\begin{thm}[Guruswami--Kopparty]\label{thm:GKMC}
{\rm\cite[Theorem~17]{VGSK}}
For every $\ff_q$-linear subspace $W\le\ff_q^{d+1}$ with $\dim(W)=s$
we have
\begin{equation*}
\sum_{\alpha\in\ff_q}\dim(H_\alpha\cap W)\le\frac{d\cdot s}{t-s+1}
\end{equation*}
\end{thm}

We do not repeat the proof here, for details see~\cite[pages~11--12]{VGSK}. The proof of this theorem uses the the following matrix. Let $W\le\ff^{d+1}_q$ be a subspace and let the polynomials $P_1,\dots,P_s$ constitute a basis of $W$. Define the following $t\times s$ matrix of polynomials:
\begin{equation*}
M(X)=\left[\begin{matrix}
P_1(X) & \dots & P_s(X)\\
P'_1(X) & \dots & P'_s(X)\\
\vdots & \ddots & \vdots \\
P^{(t-1)}_1(X) & \dots & P^{(t-1)}_s(X)\\
\end{matrix}\right]
\end{equation*}
Let $A(X)$ be the top $s\times s$ submatrix of $M(X)$ and let $L(X)$ be the determinant of $A(X)$.

The term $d\cdot s$ in the parameter $\frac{d\cdot s}{t-s+1}$ in Theorem~\ref{thm:GKMC} above comes from the fact $\deg(L(X))\le ds$. As in the previous subsection, there is again a better bound for this degree:

\begin{lem}
The polynomial $L(X)$ has degree at most $s(d-s+1)$.
\end{lem}
\begin{proof}
Expanding the determinant $L(X)=\sum_{\pi\in S_s}(-1)^{I(\pi)}\prod_{k=1}^sP^{(k-1)}_{\pi(k)}(X)$, each term $\prod_{k=1}^sP^{(k-1)}_{\pi(k)}(X)$ has degree $\sum_{k=1}^s\big(\deg(P_{\pi(k)})-(k-1)\big)$, that is equal to $\sum_{k=1}^s\deg(P_k)-\binom{s}{2}$. The basis $P_1,\dots,P_s$ of the subspace $W\le\ff^{d+1}_q$ can be chosen (by Gau{ss}ian elimination) such that $\deg(P_1)<\deg(P_2)<\dots<\deg(P_s)\le d$ and thus,
\begin{align*}
\deg(L)\le&\left(\sum_i\deg(P_i)\right)-\binom{s}{2}\le\left(d+\dots+\big(d-(s-1)\big)\right)-\binom{s}{2}=\\
=&\left(sd-\binom{s}{2}\right)-\binom{s}{2}=ds-2\frac{s(s-1)}{2}=s(d-s+1).
\end{align*}
\end{proof}

As a consequence, the Guruswami--Kopparty Theorem~\ref{thm:GKMC} above will have the following improved form.

\begin{cor}[Guruswami--Kopparty; improved]
For every $\ff_q$-linear subspace $W\le\ff_q^{d+1}$ with $\dim(W)=s$ we have
\begin{equation*}
\sum_{\alpha\in\ff_q}\dim(H_\alpha\cap W)\le\frac{(d-s+1)s}{t-s+1}
\end{equation*}
\end{cor}

These stronger versions of~\cite[Theorem~14]{VGSK} and \cite[Theorem~17]{VGSK} stated in this and the previous subsection implies a stronger version for the main~\cite[Theorem~7]{VGSK} as follows.

\begin{thm}[Guruswami--Kopparty; improved]
For all positive integers $s,r,t,m=d+1$ and prime powers $q$ satisfying $s\le t\le m<q$, there is an explicit collection of $M=\Omega\left(\frac{q^r}{rt}\right)$ linear subspaces $H_1,\dots,H_M\subset\ff_q^m$, each of co-dimension $rt$, which forms a strong $\left(s,A\right)$ subspace design, where $A\le\frac{\left(m-1-\frac{s-1}{2}\right)s}{r(t-s+1)}$, and even $A\le\frac{(m-s)s}{r(t-s+1)}$ if $m<\ch\ff_q$.
\end{thm}

So, we have shown that over large enough characteristic, the construction of~\cite[Section~5]{VGSK} based on multiplicity codes actually give us a strong $(2,2d-2)$ subspace design, and thus, a set of $2d-1$ lines in higgledy-piggledy position.

\section{Open questions}

As we have seen previously, subspace designs constructed by {Guruswami and Kopparty}~\cite{VGSK} can also give us hyperplane-generating set of lines of size $2d-1$ (if $\ch\ff>d+1$), the optimal size over algebraically closed field. But examples in low dimensions show that much smaller hyperplane-generating sets of lines could exist, if the field is finite.

\paragraph{Open problem 1}
While hunting for weak and strong $(s,A)$ subspace designs aims subspace designs of cardinality as large as possible (while $s$ and $A$ are constants), our problem is to find as small as possible hyperplane-generating sets of lines, which are weak $(2,N\!-\!1)$ subspace designs of cardinality $N$ where $N$ is as small as possible. In this article we have proved that (if the field $\ff$ has at least $1.5d$ elements, then) a generator set $\hL$ of lines in $\PG(d,\ff)$ has to contain at least $\left\lfloor\frac{d}{2}\right\rfloor+d$ elements. Open problem is to find minimal size hyperplane-generating sets of lines over fields that are not algebraically closed.

\paragraph{Open problem 2}
A natural generalization of the hyperplane-generating sets of lines would be the following. A set $\hL$ of $k$ subspaces is said to be \emph{generating set} (or set of $k$ subspaces in `higgledy-piggledy' position) if each subspace $H$ of \emph{co-dimension $k$} meet $\hL$ in a set of points that generates $H$. Open question is the minimal size of a set of $k$ subspaces in `higgledy-piggledy' position.

\paragraph{Open problem 3}
We have shown that the Guruswami--Kopparty construction based on multiplicity codes gives stronger results than the construction based on Folded Reed--Solomon codes, in case $m<\ch\ff_q$. We conjecture that using the generalization of our trick of `diverting' the tangents of the moment curve (shown in Subsection~\ref{subs:div}), can generalize this Guruswami--Kopparty constructions over small characteristics, and thus, the main Guruswami--Kopparty Theorem can be improved over fields of small characteristics.


\begin{thebibliography}{9}


\bibitem{HTP}
{\sc Tam{\'a}s~H{\'e}ger, Bal{\'a}{zs}~Patk{\'o}s and Marcella~Tak{\'a}{ts}:}
Search Problems in Vector Spaces,
\newblock {\em Designs, Codes and Cryptography}
\newblock 2014
\newblock {\em accepted}


\bibitem{SCH}
{\sc Laurent~Manivel} (Author);
{\sc John~R.~Swallow} (Translator):
\newblock {\em Symmetric Functions, Schubert Polynomials and Degeneracy Loci}.
\newblock SMF/AMS Texts and Monographs, \textbf{6.} Cours Sp{\'e}cialis{\'e}s [Specialized Courses], \textbf{3.}
\newblock American Mathematical Society; Soci{\'e}t{\'e} Math{\'e}matique de France, Paris, 2001.


\bibitem{VGSK}
{\sc Venkatesan~Guruswami and Swastik~Kopparty:}
Explicit Subspace Designs
\newblock {\em HPI ECCC Electronic Colloquium on Computational Complexity}
\newblock 10th April 2013
\newblock \url{http://eccc.hpi-web.de/report/2013/060/}


\end{thebibliography}
\end{document}